\documentclass[12pt,reqno]{amsart}
\usepackage{graphicx}
\usepackage{psfrag}
\usepackage{pxfonts}
\usepackage{mathrsfs}
\usepackage{color}
\usepackage{amsmath,amssymb}

\numberwithin{equation}{section}

\theoremstyle{plain}
\newtheorem{maintheorem}{Theorem}

\newcommand{\Q}{\mathbb{Q}}
\newcommand{\R}{\mathbb{R}}
\newcommand{\N}{\mathbb{N}}
\newcommand{\Z}{\mathbb{Z}}

\newcommand{\T}{\mathbb{T}}

\newcommand{\vol}{\mathrm{Vol}}
\newcommand{\jac}{\mathrm{Jac}}

\newtheorem{theorem}{Theorem}[section]

\newtheorem{corollary}[theorem]{Corollary}
\newtheorem{proposition}[theorem]{Proposition}
\newtheorem{lemma}[theorem]{Lemma}
\newtheorem{definition}[theorem]{Definition}

\theoremstyle{remark}
\newtheorem{remark}[theorem]{Remark}

\begin{document}

\thanks{}

\author{J. Santana C Costa}
\address{DEMAT-UFMA S\~{a}o Lu\'{i}s-MA, Brazil.}
\email{jsc.costa@ufma.br}

\author{F. Micena}
\address{
  IMC-UNIFEI Itajub\'{a}-MG, Brazil.}
\email{fpmicena82@unifei.edu.br}


\renewcommand{\subjclassname}{\textup{2000} Mathematics Subject Classification}

\date{\today}

\setcounter{tocdepth}{2}
\title{On Lyapunov exponents properties of special Anosov endomorphisms on $\mathbb{T}^d$}
\maketitle
\begin{abstract}
This work is addressed to study Anosov endomorphisms of $\mathbb{T}^d,$ $d\geq 3.$ We are interested in obtaining metric and topological information on such Anosov endomorphism by comparison between their Lyapunov exponents and the ones of its linearization. We can characterize when a weak unstable foliation of a special Anosov endomorphism near in linear is an absolutely continuous foliation. Also, we show that in dimension $d \geq 3,$ it is possible to find a smooth special Anosov endomorphism being conservative but not Lipschitz conjugate with its linearization, in contrast with the smooth rigidity in dimension two.
\end{abstract}

\section{Introduction}

In the 1970s two works Ma\~n\'e-Pugh  \cite{MP75} and Przytycki \cite{PRZ} generalized the notion of Anosov diffeomorphism for non-invertible maps, introducing the notion of Anosov endomorphism. They show that there is a big difference among these notions about structural stability, this is, they show that the Anosov endomorphism does not have structural stability as the Anosov diffeomorphism. The main reason for this is that as in Anosov endomorphisms, there are several pre-orbits, then at the point, there can be more than one unstable direction.

Let $M$ be a $C^{\infty}$ closed manifold and $f:M\rightarrow M$ be a $C^1$ local diffeomorphism.  We say that $f$ is an Anosov endomorphism if for every $(x_n)_{n \in \mathbb{Z}},$ an $f-$orbit, this is $f(x_n)=x_{n+1}$,  there is a splitting

$$T_{x_i} M = E^s_{x_i} \oplus E^u_{x_i}, \forall i \in \mathbb{Z},$$
such that $Df$ is uniformly expanding on $E^u$ and uniformly contracting on $E^s.$

The direction $E^s$ is well defined, but we can have  $(x_n)_{n \in \mathbb{Z}}$ and  $(y_n)_{n \in \mathbb{Z}}$ two different orbits of $f$ with $x_0=y_0$ such that $E^u_f(x_0)\neq E^u_f(y_0) $.   By \cite{PRZ}, these directions are integrable to unstable local discs, so at a point $x$ we can have more than one local unstable manifold. When the Anosov endomorphism has only one direction $E^u$ in each point $x\in M$ is called a special Anosov endomorphism, for instance, all linear Anosov endomorphism is special.

The works \cite{F70} and \cite{M74} show that all Anosov diffeomorphism of torus ${\T}^n$ is conjugated to a linear Anosov. A similar result there is for Anosov endomorphism, In \cite{MT19} and \cite{sumi} show that an Anosov endomorphism $f$ of torus ${\T}^n$ is conjugated to a linear Anosov if, and only if, $f$ is special.

One of our main goals is to relate the behavior of the Lyapunov exponents of Anosov endomorphisms with the Lyapunov exponents of its linearization. The Lyapunov exponents play an important role in the ergodic theory and dynamical systems. They are a useful tool of the Pesin theory, in the study of
entropy, and equilibrium states among others. The existence of these exponents is guaranteed by celebrated Osceledec's Theorem.

Given $A:{\T}^d\rightarrow {\T}^d$  a linear Anosov endomorphism which is irreducible over ${\Q}$ such that $\dim E^s_A=1$ and  $E^u_A=E^u_1\oplus E^u_2\oplus\cdots \oplus E^u_{d-1}$ we consider $f:{\T}^d\rightarrow {\T}^d$  a $C^2$  Anosov endomorphism conjugated to $A$ by  $h:{\T}^d\rightarrow {\T}^d$, that is, $h \circ A = f \circ h.$ We denote $\tilde{m} = h_{\ast}(m).$ Since $m$ is ergodic for $A$ and $h$ is conjugacy, the measure $\tilde{m}$ is ergodic for $f.$  Fixed $i\geq 1,$ we will see in the preliminaries that are defined intermediate Lyapunov exponents $ \lambda^u_i(f,x),$ for $\tilde{m}-$a.e $x \in \mathbb{T}^d.$ Since $\tilde{m}$ is ergodic the value $\lambda^u_i(f,x)$ is $\tilde{m}-$a.e constant and it will be denoted by $\lambda^u_i(f, \tilde{m}).$ We show that


\begin{maintheorem}\label{TeoA}
	Suppose that  $A:{\T}^d\rightarrow {\T}^d$ be a linear Anosov endomorphism, 
	then there is a $C^1$ neighborhood $\mathcal{U}$ of $A$, such that for every $C^r,\, r\geq 2,$ Anosov endomorphism $f\in\mathcal{U}$ conjugated to $A$, either the set $Z=\{x\in{\T}^d;\, \sum_{i = 1}^{d-1} \lambda^u_i(f,x) =  \sum_{i = 1}^{d-1} \lambda^u_i(f, \tilde{m})\}$
	meets every unstable leaf in a Lebesgue null set of the leaf or
	$$
	\displaystyle\sum_{i = 1}^{d-1} \lambda^u_i(f, \tilde{m}) =  \sum_{i = 1}^{d-1} \lambda^u_i(A).
	$$
In the last case, the conjugacy $h$ is absolutely continuous.
\end{maintheorem}

A very researched common question is about the regularity of the conjugacy between $f$ and $A$. Related to above Theorem, in  \cite{Mic22measurable} show that

\begin{theorem}[\cite{Mic22measurable}]\label{TeoM22}
	Let  $f:{\T}^2\rightarrow {\T}^2$ be an smooth Anosov endomorphism with linearization $A$,
	such that $f$ and $A$ are conjugated by $h$, for which $h\circ f=A\circ h$.
	Consider $Z=\{x\in{\T}^2;\, \lambda_f^u(x)=\lambda^u_f(\tilde{m})\}$, then either $Z$
	meets every unstable leaf in a Lebesgue null set of the leaf, or $Z={\T}^2$, and in this
	last case, $f$ and $A$ are smoothly conjugated.
\end{theorem}

\begin{theorem}[\cite{Mic22measurable}]\label{Cor01}
	Let  $f:{\T}^2\rightarrow {\T}^2$ be an smooth Anosov endomorphism with linearization $A$, such that $f$ and $A$ are conjugated. If $f$ is volume preserving, then the conjugacy is $C^{\infty}$.
\end{theorem}

One natural question is if is possible to obtain smooth conjugacy in Theorem \ref{TeoA} as we have in the Theorem \ref{TeoM22} and  \ref{Cor01}. We show that this is not possible in dimensions greater than 2 (Theorem \ref{TeoB} below). We observe that it shows that there is a substantial difference in the rigidity context when one passes from dimension two to higher dimensions.

\begin{maintheorem}\label{TeoB}
		There is a smooth Anosov endomorphism $f:{\T}^3\rightarrow {\T}^3$ conjugated with its linearization $A$ such that the set  $Z=\{x\in{\T}^3;\, \lambda_f^{wu}(x)+\lambda_f^{su}(x)=\lambda^{wu}_f(\tilde{m})+\lambda^{su}_f(\tilde{m})\}$ intersect some unstable leaf in a set positive Lebesgue measure of the leaf, moreover the conjugacy $h$ is absolutely continuous but the conjugacy \mbox{$h$ is not $C^1$}, in fact, $h$ is not even Lipschitz.
\end{maintheorem}

The next theorem establishes a relation between the absolute continuity of intermediate invariant expanding foliations and the sum of the corresponding Lyapunov exponents. In the preliminaries section, we will describe in detail the definition of intermediate bundles $E^{wu}_{f,1,k}.$

\begin{maintheorem}\label{TeoC}
Let  $A:{\T}^d\rightarrow {\T}^d$ be a linear Anosov endomorphism which is irreducible over ${\Q}$ such that $\dim E^s_A=1$ and  $E^u_A=E^u_1\oplus E^u_2\oplus\cdots \oplus E^u_{d-1},$ a direct and dominated splitting one-dimensional bundles. Then there is a $C^1$ neighborhood $\mathcal{U}$ of $A$, such that for any $f\in\mathcal{U}$ a special Anosov endomorphism is defined an expanding quasi-isometric foliation $W^{wu}_{1,k},$\, $k=1,\ldots,d-1$ tangent to $k-$dimensional intermediate and $Df-$expanding bundle $E^{wu}_{f,1,k}.$ Moreover if $f\in\mathcal{U}$ is a $C^r$, $r\geq 2$, a conservative and special Anosov endomorphism, $W^{wu}_{1,k}$ is absolutely continuous if and only if
$$
\displaystyle\sum_{i = 1}^{k} \lambda^u_i(f, m) =  \sum_{i = 1}^{k} \lambda^u_i(A).
$$

\end{maintheorem}

\section{Preliminaries}

Our results are addressed to the study of Anosov endomorphisms, see \cite{PRZ} and \cite{MP75} for detailed definitions and properties. Here, we need to explore a more fine structure of Anosov endomorphisms requiring understanding, in some sense, of the intermediate bundles of their hyperbolic splitting. We need in our case to see Anosov endomorphisms as partially hyperbolic endomorphisms. Consider $M$ a $C^{\infty}$ closed Riemannian manifold.

\begin{definition}[Partially Hyperbolic Endomoprhism \cite{CM22}]\label{def 1}
Let $f:M\rightarrow M$ be a $C^1$ local diffeomorphism. We say that $f$ is a \emph{partially hyperbolic endomorphism} if there is a Riemannian metric
$\langle\cdot,\cdot\rangle$ and
constants \mbox{$0<\nu<\gamma_1\leq\gamma_2<\mu$} with $\nu<1,$ $\mu>1$ and $C>1$ such
that for each orbit $(x_n)_{n\in{\Z}}$ of $f,$ this is $f(x_n)=x_{n+1},$ there is a decomposition
$$T_{x_n}M=E^s_f({x_n})\oplus E^c_f(x_n)\oplus E^u_f(x_n)$$ satisfying:

\begin{enumerate}

\item $Df_{x_i}(E^{\ast}_f(x_i))=E^{\ast}_f(x_{i+1}), \ast \in \{s,c,u\},$  for any $i \in \mathbb{Z},$
\item $||Df^n_{x_i}(v^s)||\leq C\nu^n||v^s||,$
 \item  $C^{-1}\gamma^n_1||v^c||\leq||Df^n_{x_i}(v^c)||\leq C\gamma^n_2||v^c||,$
 \item $C^{-1}\mu^n||v^u||\leq||Df^n_{x_i}(v^u)||,$

\end{enumerate}
for any $n \geq 0,$  $i \in \mathbb{Z},$  $ v^s\in E_{x_i}^s,  v^c\in E_{x_i}^c$ and  $v^u\in E_{x_i}^u.$
\end{definition}

This definition is also called absolute partially hyperbolic endomorphism. Note that, when $E^c$ is trivial, the above definition becomes the definition of Anosov endomorphisms.

%
%
%
%
%

There are two common ways to study endomorphism as invertible systems. One is understanding the lift in the universal cover, and the other is considering the inverse limit space. About the lift in the universal cove, we have:

\begin{proposition}[\cite{MP75, CM22}]\label{Prop-Mane}
	Let $\widetilde{M}$ be the universal cover of $M$, then $f$ is an Anosov (partially hyperbolic) endomorphism of $M$, if and only if, its lift $\tilde{f}: \widetilde{M} \rightarrow  \widetilde{M}$ is an Anosov (partially hyperbolic) diffeomorphism.
\end{proposition}

Now, let us describe the $f$  in the inverse limit space or also known as the natural extension of $f$. Let

$$
M^f=\left\{ (x_n)_{n\in {\Z}} \in  \prod_{i\in{\Z}} M_i; \,\,\, M_i=M \,\, {\rm and} \,\, f(x_n)=x_{n+1} \right\}.
$$

Since $M$ is compact and $f$ is continuous, it is easy to check that $M^f$ is a compact metric space with the metric.

$$
\hat{d}(\tilde{x},\tilde{y})=\sum_{i\in{\Z}}\frac{d(x_i,y_i)}{2^{|i|}}.
$$

We define $\hat{f}:M^f\rightarrow M^f$ by  $\hat{f}((x_n)_{n\in {\Z}})=(x_{k})_{k\in {\Z}},$ where $k=n+1,$ like shift map. Consider $p:M^f\rightarrow M,$ the projection in the zero coordinate, this is, if $\tilde{x}=(x_n)_{n\in {\Z}}$ then
$p(\tilde{x})=x_0$.  It is easy to verify that $p$ is continuous.

Partially hyperbolic endomorphisms can present a complicated structure of the center and unstable bundles. Given a point $x\in M$ can be defined infinitely many $E^c, E^u$ bundles at $x,$ see for instance Theorem A of \cite{CM22}. Nicely $E^{cs}$ is uniquely defined for each point. For more topological properties of partially hyperbolic and Anosov endomorphisms we refer to \cite{CM22} and \cite{MT16}.

\subsection{Lyapunov Exponents}

Lyapunov exponents are important constants and measure the asymptotic behavior of dynamics in the tangent space level. Let $f: M \rightarrow M$ be a measure preserving diffeomorphism. Then by  Oseledec's Theorem, for almost every $x \in M $ and any $v \in T_x M $ the following limit exists:

$$\lambda(f,x,v) := \lim_{n \rightarrow +\infty} \frac{1}{n} \log ||Df^n(x) \cdot v ||$$
and is called of the Lyapunov exponents of $f$ in $x$ in the direction $v$.

In the setting of partially hyperbolic endomorphism is important to understand the notion of intermediate Lyapunov exponent. Considering the decomposition $E^s(x_0) \oplus E^c(x_0) \oplus E^u(x_0),$ for some $(x_n)_{n \in \mathbb{Z}} $ a complete orbit of limit inverse space, it is necessary comprehension of center Lyapunov exponent and that it depends just on the $x_0.$

Consider $f$ and $A$ as in Theorem \ref{TeoC}. Denote by $\tilde{f}, \tilde{A}: \mathbb{R}^d \rightarrow \mathbb{R}^d$ the lifts of $f$ and $A$ respectively. Since $f$ is enough $C^1-$close to $A,$ then $\tilde{f}$ and $\tilde{A}$ are enough $C^1-$close. 
By Proposition \ref{Prop-Mane} $\tilde{f}$ is a partially hyperbolic diffeomorphism.  Moreover by \cite{pesin2004lectures}, we get $E^u_{\tilde{f}} = E^u_{\tilde{f},1} \oplus \ldots \oplus E^u_{\tilde{f},d-1} .$ Denote by $E^{wu}_{\tilde{f}, 1, k}  = E^u_{\tilde{f},1} \oplus \ldots \oplus E^u_{\tilde{f},k} $ and $E^{su}_{\tilde{f}, k+1, d-1}  = E^u_{\tilde{f},k+1} \oplus \ldots \oplus E^u_{\tilde{f},d-1} .$

Considering $F^c_{\tilde{f}} = E^{wu}_{\tilde{f}, 1, k} $ and $F^u_{\tilde{f}} = E^{su}_{\tilde{f}, k+1, d-1},$ bundles of the partially hyperbolic splitting of $\tilde{f}.$ Since $f$ is special and $F^c = F^{cs} \cap F^{cu},$ by Lemma 2.5 and Proposition 2.10 of \cite{CM22}, we can see that, for each $x \in \mathbb{T}^d,$ it is uniquely defined the center bundle $F^c_{f}(x) = E^{wu}_{f, 1, k}(x), $ for every $k \geq 1$ and $k+1 \leq d-1.$

If we consider for the same $x \in \mathbb{T}^d,$ two intermediate directions $E^u_{f,i}(x)$ and $\tilde{E}^u_{f,i}(x).$ Since $E^u_{f,1,i}(x)$ is uniquely well defined for $x,$ we can note that \mbox{every} \mbox{$v \in  E^u_{f,i}(x) \setminus \{0\}$} has a nontrivial component in $\tilde{E}^u_{f,i}(x)$ and vice-versa. By  \mbox{Oseledec's} Theorem, if $v \in  E^u_{f,i}(x) \setminus \{0\} $ and $\tilde{v} \in  \tilde{E}^u_{f,i}(x) \setminus \{0\}, $ then we obtain \mbox{$\lambda(f, x , v) = \lambda(f, x , \tilde{v}) $} and it will be denoted by $\lambda^u_{i}(f,x).$

\subsection{Absolute Continuity}

Roughly speaking a foliation $ W$ of $M$ is absolutely continuous if satisfies: Given a set $Z \subset M,$ such that $Z$ intersects the leaf $ W(x)$ on a zero measure  set of the leaf, with $x$ along a full Lebesgue set of $M,$ then $Z$ is a zero measure set of $M.$
More precisely we write:

\begin{definition} We say that a foliation $ W$ of $M$ is absolutely continuous if given any $ W-$foliated box $B$ and a Lebesgue measurable set $Z,$ such that $Leb_{ W(x)\cap B} ( W(x)\cap Z) = 0,$ for $m_B-$ almost everywhere $x \in B,$ then $m_B(Z) = 0.$ Here $m_B $ denotes the Lebesgue measure on $B$ and $Leb_{ W(x)\cap B} $ is the Lebesgue measure of the submanifold $ W(x)$ restricted to $B.$
\end{definition}

It means that if $P$ is such that $m_B(P) > 0,$ then there are a measurable subset $B' \subset B,$ such that $m_B(B') > 0$ and  $Leb_{ W(x)\cap B}( W(x) \cap Z) > 0$ for every $x \in B'.$

\subsection{Highlights of SRB measures for endomorphisms} The reader more acquainted with SRB theory can jump straight to the proofs of the theorems.

At this moment we need to work with the concept of SRB measures for endomorphisms. In fact, SRB measures play an important role in the ergodic theory of differentiable dynamical systems. For $C^{1+\alpha}-$systems these measures can be characterized as ones that realize the Pesin Formula or equivalently the measures for which the conditional measures are absolutely continuous w.r.t. Lebesgue restricted to local stable/unstable manifolds. We go to focus our attention on the endomorphism case. Before proceeding with the proof let us give important and useful definitions and results concerning SRB measures for endomorphisms.

First, let us recall an important result.

\begin{theorem}[\cite{QXZ}] Let $(M,d)$ be a compact metric space and $f: M \rightarrow M$ a continuous map. If $\mu$ is an $f-$invariant Borelian probability measure, the exist a unique $\tilde{f}-$invariant Borelian probability measure $\tilde{\mu}$ on $M^f,$ such that $\mu(B) = \hat{\mu}(p^{-1}(B)).$
\end{theorem}

\begin{definition}A measurable partition $\eta$ of $M^f$ is said to be subordinate to
$W^u-$manifolds of a system $(f, \mu)$ if for $\hat{\mu}$-a.e. $\tilde{x} \in M^f,$ the atom $\eta(\tilde{x}),$ containing $\tilde{x},$ has the following properties:
\begin{enumerate}
\item $p|_{\eta(\tilde{x})}: \eta(\tilde{x})  \rightarrow p(\eta(\tilde{x}))$ is bijective;
\item There exists a $k(\tilde{x})-$dimensional $C^1-$embedded submanifold $W(\tilde{x})$ of $M$ such that $W(\tilde{x}) \subset W^u(\tilde{x}),$
$$p(\eta(\tilde{x})) \subset W(\tilde{x})$$
and $p(\eta(\tilde{x}))$ contains an open neighborhood of $x_0 $ in $W(\tilde{x}).$ This neighborhood
being taken in the topology of $W(\tilde{x})$ as a submanifold of $M.$
\end{enumerate}
\end{definition}

We observe that by Proposition 3.2 of \cite{QZ}, every there exist such partition for Anosov endomorphism and they are can be taken increasing, that means $\eta$ refines $\tilde{f}(\eta) .$ Particularly $ p(\eta(\tilde{f}(\tilde{x}))) \subset p(\tilde{f}(\eta(\tilde{x}))) .$

\begin{definition} Let $f: M \rightarrow M$ be a $C^2-$endomorphism preserving an invariant Borelian probability $\nu.$
We say that $\nu$ has SRB property if for every measurable partition $\eta $ of $M^f$ subordinate
to $W^u-$manifolds of $f$ with respect to $\nu$, we have $p(\hat{\nu}_{\eta{(\tilde{x})}}) \ll m^u_{p(\eta(\tilde{x}))},$ for $\hat{\nu}-$a.e. $\tilde{x}$, where
$\{\hat{\nu}_{\eta{(\tilde{x})}} \}_{\tilde{x} \in M^f}$
is a canonical system of conditional measures of $\hat{\nu}$ associated with $\eta,$
and $m^u_{p(\eta(\tilde{x}))} $ is the Lebesgue measure on $W^u(\tilde{x})$ induced by its inherited Riemannian metric as a submanifold of $M.$
\end{definition}

The next theorems characterize SRB-property via Pesin's formula.
\begin{theorem}[Pesin Entropy Formula for Endomorphisms {\cite{PDL1}}]\label{pesin1} Let $f : M \rightarrow M$ be a $C^2$ endomorphism and $\mu$ an
$f-$invariant Borel probability measure on $M.$ If $\mu \ll m,$ then there holds
Pesin's formula
\begin{equation}
h_{\mu}(f) = \displaystyle\int_M \displaystyle\sum \lambda^i(x)^{+}m_i(x) d\mu.
\end{equation}
\end{theorem}

\begin{theorem}[\cite{QZ}] \label{pesin2} Let $f$ be a $C^2$ endomorphism on $M$ with an invariant Borel probability
measure $\mu$ such that $\log(|Jf(x)|) \in L^1(M,\mu).$ Then the entropy
formula
\begin{equation}\label{PesinU}
h_{\mu}(f) = \displaystyle\int_M \displaystyle\sum \lambda^i(x)^{+}m_i(x) d\mu
\end{equation}
holds if and only if $\mu$ has SRB property.
\end{theorem}

Similar statements can be obtained analogously for invariant expanding foliations.

\subsection{Quasi-isometry and linearization}

An important tool that we need here is the quasi-isometry of foliations.

\begin{definition}
A foliation $W$ of a closed manifold $M$ is called quasi-isometric if there is a
constant $Q > 0$, such that in the universal cover $\widetilde{M}$ we have:
$$
d_{\widetilde{W}}(x,y)\leq Q d_{\widetilde{M}}(x,y) +Q
$$
for every $x, y$ points in the same lifted leaf $\widetilde{W}$,
where $\widetilde{W}$ denotes the lift of $W$ on $\widetilde{M}$.
\end{definition}

Here $d_{\widetilde{W}}$ denotes the Riemannian metric on $\widetilde{W}$ and $d_{\widetilde{M}}$ is a
Riemannian metric of the ambient $\widetilde{M}$.

From Brin, \cite{B} we have.

\begin{proposition}[\cite{B}]\label{B1} Let $W$ be a $k-$dimensional foliation of $\mathbb{R}^d.$ Suppose that there is an $(d - k)$ dimensional plane $P$ such that $T_x W(x)\cap P = \{0\}$ and $\angle(T_x W(x), P) > \beta > 0$ for every $x \in \mathbb{R}^d.$ Then $W$ is quasi-isometric.
\end{proposition}

Another important tool we will need is linearization. Every endomorphism $f$ of the torus ${\T}^d$ induces an automorphism of the fundamental group 
and there exists a unique linear endomorphism $f_{\ast}$ which induces the same automorphism
on $\pi_1({\T}^d)$. The endomorphism $f_{\ast}$ is called the linearization of $f$. When
$f$ is partially hyperbolic its linearization is also partially hyperbolic.  Similary as in \cite{H},  below we have results that  comes from the fact that $f$ and its linearization are homotopic and by Proposition \ref{Prop-Mane},  in the universal cover $\R^d$, its are diffeomorfisms.

\begin{proposition}\label{Prop Hamm 1}
Let $f:{\T}^d\rightarrow {\T}^d$ be a partially hyperbolic endomorphism with linearization $A$,
then for each $n\in{\Z}$ and $C>1$ there is an $M>0$ such that for all $x,y\in {\R}^d$ and
$$
||x-y||>M\,\,\, \Rightarrow\,\,\,\frac{1}{C}<\frac{||\tilde{f}^n(x)-\tilde{f}^n(y)||}{||\tilde{A}^n(x)-\tilde{A}^n(y)||}<C.
$$.

\end{proposition}

For a subset $X\subset{\R}^d$ and $R>0$, led $B_R(X)$ denote the neighbourhood
$$
B_R(X)=\{y\in{\R}^d;\,\,||x-y||<R\,\,\mbox{for\,\,some}\,\,x\in X\}.
$$
\begin{proposition}[\cite{H}]\label{Prop Hamm 3}
Let $f:{\T}^d\rightarrow {\T}^d$ be a partially hyperbolic endomorphism dynamically coherent with linearization $A$,
then there is a constant $R_c$ such that for all $x\in{\R}^d$,
\begin{itemize}
\item[(1)] $\widetilde{W}^{cs}_f(x)\subset B_{R_c}(\widetilde{W}^{cs}_A(x))$,
\item[(2)] $\widetilde{W}^{cu}_f(x)\subset B_{R_c}(\widetilde{W}^{cu}_A(x))$,
\item[(3)] $\widetilde{W}^{c}_f(x)\subset B_{R_c}(\widetilde{W}^{c}_A(x))$.
\end{itemize}

\end{proposition}

\begin{corollary}[\cite{H}]\label{Cor Hamm 2}
If $||x-y||\rightarrow\infty$ where $y\in \widetilde{W}^{c}_f(x)$
then $\frac{x-y}{||x-y||}\rightarrow \widetilde{E}^{c}_A(x)$ uniformly.
More precisely, for $\varepsilon>0$ there exists $M>0$ such that if $x\in{\R}^d$, $y\in \widetilde{W}^c_f(x)$ and $||x-y||>M$, then
$$
||\pi_A^{c\perp}(x-y)||<\varepsilon||\pi^{c}_A(x-y)||.
$$
where $\pi_A^{c}$ is the orthogonal projection in the subspace $E^c_A$ and $\pi_A^{c\perp}$ is the projection in the orthogonal
subspace $E^{c\perp}_A.$
\end{corollary}

\section{Proof of Theorem \ref{TeoA}}

\begin{proposition}\label{Prop vol untable}
Let  $f:{\T}^d\rightarrow {\T}^d$ be a linear Anosov endomorphism conjugated with its linearization $A$ and there is an $\alpha>0$ such that in the universal cover $\angle(T_x\widetilde{W}^u_f, (\widetilde{E}^u_A)^{\bot})>\alpha>0$ for any $x\in{\R}^d,$
then given  $\varepsilon>0$, there is  $M > 0$ such that
$$
\vol(\tilde{f}^n\pi^{-1}(U_r))\leq C_0(1+\varepsilon)^{nd_u}e^{n\sum\lambda_i^u(A)}\vol(U_r)
$$
 for any $n,$  where $\pi$ is the  orthogonal projection from  $\widetilde{W}^u_f(z)$ to $\widetilde{E}^u_A(z)$ (parallel to $(\widetilde{E}^u_A)^{\bot}$) and $U_r\subset\widetilde{E}^u_A$ is a ball centered in $z$ of radius $r$ where $r>M$.
\end{proposition}

This proposition is similar (Proposition 3.16, \cite{CT23}) just like its proof. We will need some auxiliary results.

\begin{lemma}\label{lem01}
	Let  $f$ be as in Proposition \ref{Prop vol untable}, then the foliation $\widetilde{W}^u_f$ is quasi-isometric.
\end{lemma}

\begin{proof}
	We just need to apply the Proposition \ref{B1}, doing $P=(\widetilde{E}^u_A)^{\bot}.$ By the assumptions of Proposition \ref{Prop vol untable}, we have
	$\angle(T_x\widetilde{W}^u_f,P)>\alpha>0$, thus by the Proposition \ref{B1}, $\widetilde{W}^u_f$ is
	quasi-isometric.
\end{proof}

To prove the above proposition we need the following auxiliary results:

\begin{lemma}\label{lema Lp}
Let  $f:M\rightarrow M$ be a Lipschitzian map, then there is $K > 0$ such that $|\jac f(x)|\leq K,$ for any $x\in M$ such that  $f$ is differentiable at $x.$
\end{lemma}

\begin{proof}[Proof of Lemma]
Let $L$ be the the Lipschitz constant of $f$, give  $x\in M$ such that $f$ is differentiable and  $v\in T_xM$, then
$$
||D_xf(v)||=\displaystyle\lim_{t\rightarrow 0}\left\|\frac{f(x + tv)-f(x)}{t}\right\|\leq\lim_{t\rightarrow 0}\frac{L||x+tv-x||}{|t|}=L||v||
$$
It implies that  $||D_xf||=\displaystyle\sup_{v\in T_xM}\frac{||D_xf(v)||}{||v||}\leq L.$

Since $\R^{n^2}$  is isomorphic to $M_n(\R),$ the space of all $n\times n-$matrixes with real coefficients,
by equivalence between norms of $\R^{n^2}$ we have \mbox{$\{A \in M_n(\R)|\; ||A||\leq L\}$} is compact.
Since $\det: M_n(\R) \rightarrow \R $ is continuous, then there is $K \geq 0$ such that
$$
\jac f(x)=|\det D_xf| \leq K.
$$
It concludes the proof.
\end{proof}

\begin{proposition}\label{Prop recob}
Let $f:X\rightarrow Y$ be a covering map such that $X$ is path connected and $Y$ is simply connected, then $f$ is a homeomorphism.
\end{proposition}

\begin{proof}[Proof of the Proposition \ref{Prop vol untable}]

To prove it we will need some auxiliary claims:

\textbf{Claim  1:} For  $z\in{\R}^d$, the orthogonal projection \mbox{$\pi: \widetilde{W}^u_f(z)\rightarrow \widetilde{E}^u_A$}
(parallel to  $(\widetilde{E}^u_A)^{\bot}$) is a uniform bi-Lipschitz diffeomorphism.
\begin{proof}[Proof of Claim 1]
The argument used here is similar to that used in \cite{B}, Proposition 4.  Since the plane $(\widetilde{E}_A^u)^{\bot}$ is transversal to
the foliation $\widetilde{W}^u_f$ then there is  $\beta>0$ such that
		\begin{equation}\label{eq10}
		||d\pi(y)(v)||\geq\beta||v||
		\end{equation}
for any $x\in \widetilde{W}^u_f(z)$ and $v\in T_x\widetilde{W}^u_f(z)$, it implies that $\jac \pi(x)\neq 0$,
by Inverse Function Theorem  for each $x\in \widetilde{W}^u_f(z)$ there is a ball
$B(\delta,x)\subset \widetilde{W}^u_f(z)$ such that $\pi|_{B(\delta,x)}$ is a diffeomorphism. From proof of Inverse
Function Theorem and Equation \ref{eq10}, $\delta$ can be taken independent of $x$. Again, by Equation \ref{eq10}
there is $\varepsilon>0$, independent of $x$ such that $B(\varepsilon,\pi(x))\subset \pi(B(\delta,x))$. For to prove which
 $\pi$ is  surjective, we will show that $\pi(\widetilde{W}^u_f(z))$ is open and closed set. As $\pi$ is a local homeomorphism, then
 $\pi(\widetilde{W}^u_f(z))$ is open. To verify that it is also closed let
$y_n\in \pi(\widetilde{W}^u_f(z))$ be a sequence converging to $y$, hence there is a $n_0$ enough large
such that $y\in B(\varepsilon,y_{n_0})$ and therefore $y\in \pi(\widetilde{W}^u_f(z))$, then $\pi$ is surjective.
Moreover $\pi$ is a covering map, in fact for any $y\in \widetilde{E}^u_A$ there is a neighborhood
$B(\varepsilon,y)$ with $\pi^{-1}(B(\varepsilon,y))=\bigcup U_i$, where $\pi:U_i\rightarrow B(\varepsilon,y)$ is
a diffeomorphism. The injectivity follows from Proposition \ref{Prop recob}.

The map $\pi$ is Lipschitz.  In fact $||\pi(x)-\pi(y)||\leq||x-y||\leq d_{\widetilde{W}_f^u}(x,y)$. Let us show that
$\pi^{-1}$ is  Lipschitz.
From the Equation \ref{eq10},  consider $L=\frac{1}{\beta}$ such that  $||d\pi^{-1}(y)(v)||\leq L||v||$ for any
$y\in W^u_f(z)$ e $v\in T_yW^u_f(z)$. Let  $[x,y]$ be the line segment in  $\widetilde{E}^u_A$ connecting $x$ to  $y,$ so the set
$\pi^{-1}([x,y])=\gamma$ is a smooth curve connecting the points $\pi^{-1}(x)$ and $\pi^{-1}(y)$ in $\widetilde{W}^u_f(z)$. So,
$$
d_{\widetilde{W}^u}(\pi^{-1}(x),\pi^{-1}(y))\leq {length}(\gamma)=\int_{[x,y]}|d\pi^{-1}(t)|dt\leq L||x-y||.
$$
\end{proof}


\textbf{Claim 2:}  Let  $x,y\in \widetilde{E}^u_A$ , then given    $n$ and  $\varepsilon>0$, there is $M>0$ such that, if $||x-y||>M,$ then
$$
(1-\varepsilon)\nu^n||\pi^{-1}(x)-\pi^{-1}(y)||\leq||\tilde{A}^n\pi^{-1}(x)-\tilde{A}^n\pi^{-1}(y)||
\leq(1+\varepsilon)\eta^n||\pi^{-1}(x)-\pi^{-1}(y)||.
$$

\begin{proof}[Proof of Claim 2]
	We prove the second inequality, the first one shows similarly.  By Corollary \ref{Cor Hamm 2} we have
	$$
	\frac{\pi^{-1}(x)-\pi^{-1}(y)}{||\pi^{-1}(x)-\pi^{-1}(y)||}=v^u+e_M
	$$
	where  $v^u$ is an  unit vector in $\widetilde{E}^u(A)$  and  $e_M$ is a vector correction with
	converges uniformly  to  zero when  $M \rightarrow +\infty.$ It follows that
	$$
	\tilde{A}^n\left(\frac{\pi^{-1}(x)-\pi^{-1}(y)}{||\pi^{-1}(x)-\pi^{-1}(y)||}\right)\leq\eta^n v^u+\tilde{A}^ne_M=
	\eta^n\left(\frac{\pi^{-1}(x)-\pi^{-1}(y)}{||\pi^{-1}(x)-\pi^{-1}(y)||}\right)-\eta^n e_M + \tilde{A}^ne_M
	$$
	it implies that
	$$
	\frac{||\tilde{A}^n(\pi^{-1}(x)-\pi^{-1}(y))||}{||\pi^{-1}(x)-\pi^{-1}(y)||}\leq\left\|\eta^n\left(\frac{\pi^{-1}(x)-\pi^{-1}(y)}{||\pi^{-1}(x)-\pi^{-1}(y)||}\right)
	-\eta^n e_M +\tilde{A}^ne_M\right\|
	$$
	thus
	\begin{align*}
	||\tilde{A}^n(\pi^{-1}(x)-\pi^{-1}(y))||
	\leq||\pi^{-1}(x)-\pi^{-1}(y)||(\eta^n+\eta^n||e_M||+||\tilde{A}^n||||e_M||).
	\end{align*}
	Since  $n$ is fixed and  $||e_M||\rightarrow 0$ when $M\rightarrow\infty$, choose a large  $M$ such that
	$$
	\eta^n||e_M||+||\tilde{A}^n||||e_M||\leq\varepsilon\eta^n
	$$
	we conclude that
	$$
	||\tilde{A}^n\pi^{-1}(x)-\tilde{A}^n\pi^{-1}(y)||
	\leq(1+\varepsilon)\eta^n||\pi^{-1}(x)-\pi^{-1}(y)||.
	$$
\end{proof}

\textbf{Claim 3:} Give $\varepsilon>0$, there is $M > 0$ such that $\tilde{f}^n(\pi^{-1}(U_r)) \subset \pi^{-1}((1+ \varepsilon)^n\tilde{A}^n(U_r)),$ for every $r>M$ and  $n \geq 1,$ up to translation of $(1+ \varepsilon)^n\tilde{A}^n(U_r).$ The notation $(1+ \varepsilon)^n\tilde{A}^n(U_r)$ represents a homothety of the $\tilde{A}^n(U_r)$ by the factor  $(1+ \varepsilon)^n$.

\begin{proof}

For $\varepsilon>0$ let $C=1+\varepsilon$, by the Proposition \ref{Prop Hamm 1}, there is $M>0$ such that if $x,y\in \widetilde{E}^{cu}$ with 
$ ||x - y|| \geq M$ we have

\begin{equation}\label{eq023}
||\tilde{f}\pi^{-1}(x)-\tilde{f}\pi^{-1}(y)||\leq (1+\varepsilon)||
\tilde{A}\pi^{-1}(x)-\tilde{A}\pi^{-1}(y)||
\end{equation}

Let $U_r$ as in the statement. By the Claim 1 and since for all $y\in \partial U_r$ satisfies the Equation \ref{eq023}, we have

\begin{equation*}
\tilde{f}(\pi^{-1}(U_r)) \subset \pi^{-1}((1+ \varepsilon)\tilde{A}(U_r)),
\end{equation*}
up to a possible translation of $\tilde{A}(U_r),$ see Figure 1.

\begin{figure}[!htb]
	\centering
	\includegraphics[scale=0.8]{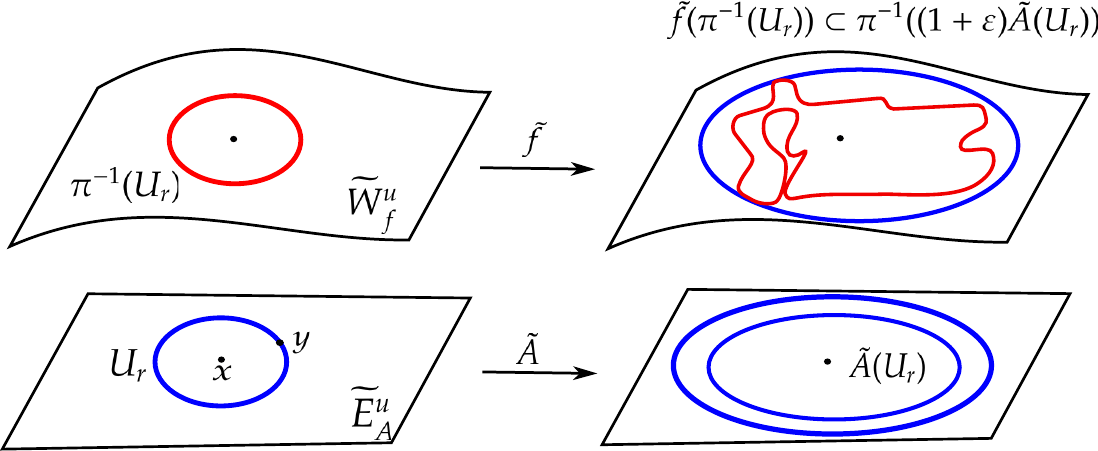}
	\caption{}\label{cone1}
\end{figure}

Now, denote $U_{r,1} = (1+ \varepsilon)\tilde{A}(U_r) .$  It is important to note that if $||x - y|| \geq M$, then
\mbox{$||\tilde{f}\pi^{-1}(x)-\tilde{f}\pi^{-1}(y)||>M$} too. Indeed, we can take $M$ satisfying also the Proposition \ref{Prop Hamm 1} and  Claim 2, so

\begin{align*}
||\tilde{f}\pi^{-1}(x)-\tilde{f}\pi^{-1}(y)||&\geq (1+\varepsilon)^{-1}||
\tilde{A}\pi^{-1}(x)-\tilde{A}\pi^{-1}(y)||\\
&\geq (1+\varepsilon)^{-1}\mu||\pi^{-1}(x)-\pi^{-1}(y)||\\
&\geq ||x - y|| \geq M,
\end{align*}

if $\varepsilon$ is enough small.

So we can do the same process as before, for $U_{r,1},$ we obtain:
$$\tilde{f}(\pi^{-1}(U_{r,1})) \subset \pi^{-1}((1+ \varepsilon)\tilde{A}(U_{r,1})).$$

Since $\tilde{f}^2(\pi^{-1}(U_r)) \subset \tilde{f}(\pi^{-1}(U_{r,1})),$ it implies
$$\tilde{f}^2(\pi^{-1}(U_r)) \subset \pi^{-1}((1+ \varepsilon)^2\tilde{A}^2(U_r)). $$

Following this inductive argument, we conclude that

\begin{equation}\label{contem}
\tilde{f}^n(\pi^{-1}(U_r)) \subset \pi^{-1}((1+ \varepsilon)^n\tilde{A}^n(U_r)),
\end{equation}
for every $n \geq 1.$

 \end{proof}

By Lemma \ref{lema Lp} follows that
\begin{align*}
\vol(\tilde{f}^n\pi^{-1}(U_r))&\leq\vol(\pi^{-1}((1+\varepsilon)^{2n}\tilde{A}^n(U_r))= C_0(1+ \varepsilon)^{nd_u} e^{n \sum_{i = 1}^{d_u} \lambda^u_i(A)}\vol(U_r),
\end{align*}
The constant  $C_0,$ comes from Lemma \ref{lema Lp}. It concludes the proof of the Proposition \ref{Prop vol untable}.
\end{proof}

\begin{proof}[Proof of Theorem \ref{TeoA}]

Suppose that  $Z$ meets some unstable leaf in a positive Lebesgue measure set of the leaf and we will prove the desired equality.

First we  prove that $\displaystyle\sum_{i = 1}^{d-1} \lambda^u_i(f, \tilde{m}) \leq  \sum_{i = 1}^{d-1} \lambda^u_i(A).$ Suppose by contradiction which $\displaystyle\sum_{i = 1}^{d-1} \lambda^u_i(f, \tilde{m}) >  \sum_{i = 1}^{d-1} \lambda^u_i(A).$

Denote by $Z^u_p=Z\cap W^u_f(p)$, where $W^u_f(p)$ is the leaf which meets $Z$ in a positive Lebesgue measure and denoted by ``Vol"  the volume  $d-1$ of the leaves $W^u_f$. So, Vol$(Z^u_p)>0$.

Let $P:{\R}^d\rightarrow {\T}^d$ be the covering map, $D\subset{\R}^d$ a fundamental domain  and denote \mbox{$\widetilde{Z^u_p}=P^{-1}(Z^u_p)\cap D$}. Since $D\tilde{f}^n$ and $Df^n$ are conjugate matrices,  \mbox{$D\tilde{f}^n=DP^{-1}\circ Df^n\circ DP$}, it follows that
$\sum_i\lambda_i^u(\tilde{f},x)>\sum_i\lambda_i^u(A)$ for any $x\in \widetilde{Z^u_p}$.

For each $q\in {\N}-\{0\}$ we define the set
$$
Z_q=\left\{x\in \widetilde{Z^u_p};\sum_i\lambda_i^u(\tilde{f},x)>\sum_i\lambda_i^u(A)+\log\left(1+\frac{1}{q}\right)\right\}.
$$
Clearly $\bigcup_{q=1}^{\infty}Z_q=\widetilde{Z^u_p}$, thus there is  $q$ such that Vol$(Z_q)>0$. For each $x\in Z_q$ follows that
$$
\displaystyle\lim_{n\rightarrow\infty}\frac{1}{n}\log|\jac \tilde{f}^n(x)|_{\widetilde{E}^u}|>\sum_i\lambda_i^u(A)+\log\left(1+\frac{1}{q}\right).
$$
So there is  $n_0$ such that for   $n\geq n_0$  we have
\begin{align*}
\frac{1}{n}\log|\jac \tilde{f}^n(x)|_{\widetilde{E}^u}|&>\sum_i\lambda_i^u(A)+\log\left(1+\frac{1}{q}\right)\\
                                         &>\frac{1}{n}\log e^{n\sum_i\lambda_i^u(A)}+\frac{1}{n}\log\left(1+\frac{1}{q}\right)^{n},
\end{align*}
it implies that
$$
|\jac \tilde{f}^n(x)|_{\widetilde{E}^u}|>\left(1+\frac{1}{q}\right)^{n}e^{n\sum_i\lambda_i^u(A)}.
$$
By this fact, for every $n>0$ we define
$$
Z_{q,n}=\left\{x\in Z_q;|\jac \tilde{f}^k(x)|_{\widetilde{E}^u}|>\left(1+\frac{1}{q}\right)^{k}e^{k\sum_i\lambda_i^u(A)},\,\,\forall\,\,k\geq n \right\}.
$$
There is $N>0$ with $\vol(Z_{q,N})>0,$ for some integer $q > 0.$

There is a set $\pi^{-1}(U_r)\subset \widetilde{W}^u_f(p)$, where $\pi^{-1}(U_r)$ is as in the Proposition \ref{Prop vol untable}
containing $p$ and $r$ is enough large
such that \mbox{$\vol(\pi^{-1}(U_r)\cap Z_{q,N})>0.$} Choose a   number $\varepsilon > 0$ as in Proposition \ref{Prop vol untable}, such that $(1 + \varepsilon)^{d_u} < \left(1 + \frac{1}{q}\right).$ By Proposition \ref{Prop vol untable} we have

\begin{equation}\label{eq4}
\vol(\tilde{f}^n(\pi^{-1}(U_r)))\leq C_0(1+\varepsilon)^{n d_c} e^{n\sum\lambda_i^u(A)}\vol(U_r),
\end{equation}
for any $n \geq 1.$

On the other hand, let  $\alpha>0$ be such that $\vol(\pi^{-1}(U_r)\cap Z_{q,N})=\alpha \vol(\pi^{-1}(U_r)),$ for $n > N,$ we have
\begin{align}\label{eq5}
\vol(\tilde{f}^n(\pi^{-1}(U_r)))&=\displaystyle\int_{\pi^{-1}(U_r)}|\jac \tilde{f}^n(x)|_{\widetilde{E}^u}| d \vol \nonumber\\
                &\geq\displaystyle\int_{\pi^{-1}(U_r)\cap Z_{q,N}}|\jac \tilde{f}^n(x)|_{\widetilde{E}^u}| d \vol\nonumber\\
                &>\displaystyle\int_{\pi^{-1}(U_r)\cap Z_{q,N}}\left(1+\frac{1}{q}\right)^{n}e^{n\sum_i\lambda_i^u(A)}d \vol\nonumber\\
                &>\left(1+\frac{1}{q}\right)^{n}e^{n\sum_i\lambda_i^u(A)} \vol(\pi^{-1}(U_r)\cap Z_{q,N})\nonumber\\
                &>\left(1+\frac{1}{q}\right)^{n}e^{n\sum_i\lambda_i^u(A)} \alpha\vol(\pi^{-1}(U_r))\nonumber\\
                &>\left(1+\frac{1}{q}\right)^{n}e^{n\sum_i\lambda_i^u(A)} \alpha\vol(U_r).
\end{align}
The equations $(\ref{eq4})$ and $(\ref{eq5})$ given us a contradiction when  $n$ is enough large, thus proving the desired inequality.

 We prove now the second inequality and for this, we use the known Ruelle's inequality, in fact,

$$
\sum_{i = 1}^{d-1} \lambda^u_i(A)=h_m(A)=h_{\tilde{m}}(f)\leq   \displaystyle\sum_{i = 1}^{d-1} \lambda^u_i(f, \tilde{m})
$$

Finally, since $\tilde{m}$ satisfies the Pesin entropy formula, thus $\tilde{m}$ has absolutely continuous conditional measures along $W^u_f-$leaves, see \cite{QZ}. Since the restriction of $h$ to stable leaves is $C^1,$ also $\tilde{m}$  has absolutely continuous conditional measures along $W^s_f-$leaves. Since $f$ is $C^2,$ the foliations $W^s_f$ and $W^u_f$ are absolutely continuous foliation, thus $\tilde{m}$ is absolutely continuous. The proof of the theorem is complete.

\end{proof}

\section{Proof of Theorem \ref{TeoB} }

The next lemmas are important to prove Theorems \ref{TeoB} and \ref{TeoC}.

\begin{lemma} \label{lem1}
	Let $f: M\rightarrow M$ be a transitive Anosov endomorphism with a degree
	$k\geq1$, where $M$ is a $C^{\infty}$ compact and connect Riemannian
	manifold, then $f$ preserves a $C^1$ volume form $m$ on $M$ if and only if
	$\jac f^n(p)=k^n,$ for any $p\in M,$ such that $f^n(p)=p$, with $n\geq 1$.
\end{lemma}

The proof of this Lemma can be found in \cite{Mic22measurable}.

\begin{lemma}\label{lem2}
	Let $A:{\T}^n\rightarrow {\T}^n$ be a linear Anosov endomorphism as in the statement of Theorem C. Then there is an open $C^1$ neighborhood $\mathcal{U}$ of $A$, such that for all $f\in\mathcal{U}$ is Anosov endomorphism with partially hyperbolic decomposition, furthermore if $f$ is special, then it is dynamically coherent with quasi-isometric $wu-$foliation.
\end{lemma}

\begin{proof}
 There is an open $C^1$ neighborhood $\mathcal{U}$ such that for all $f\in\mathcal{U}$ is an Anosov endomorphism with partially hyperbolic decomposition. Consider $\tilde{A}$ and $\tilde{f}$ in the universal cover ${\R}^n$. By Proposition \ref{Prop-Mane}, $\tilde{A}$ and $\tilde{f}$ are diffeomorphisms.
Given $\alpha>0$, the set $\mathcal{U}$ can be take such for all $f\in\mathcal{U}$ we have $\angle(\widetilde{E}^s_A\oplus\widetilde{E}^{wu}_A,\widetilde{E}^{su}_f)>\alpha$ and  $\angle(\widetilde{E}^u_A, \widetilde{E}^{s}_f)>\alpha$, then as the Lemma \ref{lem01} and using the Proposition \ref{B1} the foliations   $\widetilde{W}^u_f$ and $\widetilde{W}^s_f$ are quasi-isometric and by \cite{B} $\tilde{f}$ is dynamically coherent.

Now let's show that $W^{wu}_f$ is well defined. Consider $W^{wu}_f$ projected of the universal cover. We know that the foliation $W^u_f$ tangent to $E^{wu}_f\oplus E^{su}_f$ is well defined and  by Lemma 2.5 of \cite{CM22}  the bundle $E^{wu}_f\oplus E^s_f$ is uniquely defined. If there was a point $p$ admitting two different local leaves $W^{cs}_{1,f}(p)$ and $W^{cs}_{1,f}(p)$ tangent to $E^{wu}_f\oplus E^s_f$, by invariance of $E^{wu}_f\oplus E^s_f$ we could lift these local leaves to the same level and so the local leaves $\widetilde{W}^{cs}_{1,f}(p)$ and $\widetilde{W}^{cs}_{1,f}(p)$ contradicts the dynamic coherence of $\tilde{f}$,  hence $W^{cs}$ tangent to $E^{wu}_f\oplus E^s_f$ is well defined too. Then since $f$ is special the foliation $W^{wu}_f=W^{cs}_f\cap W^{u}_f$ is well defined.
\end{proof}

\begin{lemma}\label{lem3}
	Let  $\mathcal{U}$ as in Lemma \ref{lem2},  $f\in\mathcal{U}$ a special Anosov endomorphism and $\tilde{h}$ the conjugacy between $\tilde{f}$ and $\tilde{A}$ in the universal cover, then $\tilde{h}(\widetilde{W}^{\sigma}_f(x))=\widetilde{W}^{\sigma}_A(\tilde{h}(x))$ for $\sigma\in\{s,u,wu\}$.
\end{lemma}

\begin{proof}

Here we are considering $f$ with two decompositions of tangent space $TM$ of
${\T^n}.$ One is the hyperbolic decomposition $TM = E^s\oplus E^u$ and the other is partially hyperbolic decomposition $TM =E^s\oplus E^{wu}\oplus E^{su}.$
Since $\tilde{h}$ is a conjugacy between $\tilde{A}$ and $\tilde{f},$ then $\tilde{h}$ matches unstable and stable manifolds.

Now consider $E^{wu}_f = E^u_{f,1} \oplus \ldots \oplus \ldots E^u_{f,k} $ and $ W^{wu}_f = W^u_{1,k}$ the corresponding tangent foliation to $E^{wu}_f,$ where $1 \leq k < d-1.$ Let us to prove that $\tilde{h}(\widetilde{W}^{wu}_f) =  \widetilde{W}^{wu}_A,$ that is  $\tilde{h}(\widetilde{W}^{wu}_f(x))=\widetilde{W}^{wu}_A(\tilde{h}(x)),$ for every $x \in \mathbb{R}^n.$

Suppose by contradiction
which $\tilde{h}(x)\notin \widetilde{W}^{wu}_A(\tilde{h}(y))$, by the Claim 1 we have that \mbox{$\tilde{h}(x)\in \widetilde{W}^{u}_A(\tilde{h}(y))$}  then, by linearity of $A,$ we have
$$
||\tilde{h}(\tilde{f}^j(x))-\tilde{h}(\tilde{f}^j(y))||=||\tilde{A}^j(\tilde{h}(x))-\tilde{A}^j(\tilde{h}(y))||\geq  C_A \mu_A^j ||\tilde{h}(x)-\tilde{h}(x)||,
$$ where $\mu_A = \min \{ |\alpha^u_i|, i =k+1, \ldots, n  \},$ the numbers $ \alpha^u_i, i =k+1, \ldots, n$ are the eigenvalues of $\tilde{A}$ with absolute values bigger than one and $C_A$ is a constant depending on $A$ and the given $x$ and $y.$

Observe that $\tilde{h} = Id_{\mathbb{R}^n} + \varphi,$ where $\varphi: \mathbb{R}^n \rightarrow  \mathbb{R}^n$ is continuous and bounded on $\mathbb{R}^n $ and null on $\mathbb{Z}^n.$ Then there is $K > 0$ such that

$$||\tilde{h}(\bar{f}^j(x))-\tilde{h}(\bar{f}^j(y))|| \geq ||\tilde{f}^j(x)- \tilde{f}^j(y)|| - 2K \geq  C_A \mu_A^j||\tilde{h}(x)-\tilde{h}(x)|| ,$$

$$ ||\tilde{f}^j(x)-\tilde{f}^j(y)|| \geq  C_A \mu_A^j||\tilde{h}(x)-\tilde{h}(x)|| + 2K.  $$
We get

$$ d_{\tilde{f}^j(\widetilde{W}^{wu}_f(y) )} (\tilde{f}^j(x), \tilde{f}^j(y)) \geq C_A \mu_A^j||\tilde{h}(x)-\tilde{h}(x)||. $$ Since $f$ is enough $C^1-$close to $A,$ the partially hyperbolic constants of $f$ are close to the corresponding for $A.$ The last inequality, taking $j  \rightarrow +\infty,$ contradicts this fact for endomorphisms $C^1-$enough close to $A.$

\end{proof}

\begin{lemma}\label{lemaperiodico}[Theorem C of \cite{Mic22measurable}]  Let $f: \mathbb{T}^2 \rightarrow \mathbb{T}^2 $ be a $C^{\infty}$ special Anosov endomorphism, with degree $d \geq 2$ and $A$ its linearization. The following are equivalent.
\begin{enumerate}
\item $f$ preserves a measure $\mu$ absolutely continuous with respect to $m.$
\item $f$ is smoothly conjugated with its linearization $A.$
\item $f$ preserves a measure $\mu$ absolutely continuous with respect to $m,$ with $C^1$ density.
\item For any point $p $ such that $f^n(p) = p,$ for some integer $n \geq 1,$ holds $Jf^n(p) = d^n.$
\item There exists $c>0$ such that $Jf^n(p) = c^n,$ for any $p$ such that $f^n(p) = p,$ for some $n \geq 1$ is an integer number.
\end{enumerate}

\end{lemma}


For the construction in Theorem \ref{TeoB}, we need the following.

\begin{proposition}[\cite{BB}]\label{babo} Let $(M,m)$ be a compact manifold endowed with a $C^r$ volume
	form, $r \geq 2.$ Let $f$  be a $C^1$ and $m-$preserving diffeomorphisms of $M,$  admitting a dominated
	partially hyperbolic splitting $TM= E^s
	\oplus E^c \oplus E^u.$ Then there are arbitrary $C^1-$close
	and $m-$preserving perturbation $g$ of $f,$ such that
	
	$$\displaystyle \int_M \log(\jac^cg(x))dm > \displaystyle \int_M \log(\jac^cf(x))dm,$$
	where $\jac^cf(x)$ is the absolute value of the determinant of $Df$  restricted to $E^c_f(x).$
\end{proposition}

\begin{remark}
	When $f$ is $C^r, r\geq 1,$ in  Proposition \ref{babo} above  the perturbation $g$ also can be taken $C^r.$
\end{remark}

The proof of the above proposition involves local perturbations by modifications in specific directions. In our case, by Lemma \ref{lem2}, stable, center, and unstable directions are well defined for each point $x \in \mathbb{T}^3.$  So we can apply local perturbations as in Proposition \ref{babo}, in order to get an increment of the center Lyapunov exponent.

\begin{proof}[Proof of Theorem \ref{TeoB}]
Consider $A:{\T}^3\rightarrow{\T}^3$  a  non-invertible linear Anosov endomorphism  $deg(A) = d >1,$
such that  $T{\T}^3=E^s\oplus E^{wu}\oplus E^{u}.$ Here $E^c = E^{wu}.$

As in the proof of Proposition \ref{babo} in \cite{BB} it is possible to  get a $C^{\infty}-$volume preserving Anosov endomorphism  $f=A \circ \phi$ such that 
\begin{equation}\label{eq01}
\lambda^{wu}_f=\displaystyle \int_{{\T}^3} \log(\jac^{wu}f(x))dm > \displaystyle \int_{{\T}^3} \log(\jac^{wu} A)dm=\lambda^{wu}_A.
\end{equation}
Where $\phi(x^s,x^{wu},x^u)=(x^s,\phi_2(x^{wu},x^u),\phi_3(x^{wu},x^u))$ is a volume-preserving perturbation  close to identity that preserves the sub bundle $E^s$. Since $\phi$ does not change the direction $E^s$, we have $\lambda^s_A=\lambda^s_f$, in particular $\lambda^s_A(p)=\lambda^s_f(p)$ for all $p\in Per(f)$, using Theorem 1.1 of \cite{AGGS22}, it implies that $f$ is special. By \cite{MT19} we get that $f$ and $A$ are conjugated by $h.$

Since $f$ is conservative $\lambda^s_f(p) + \lambda^{wu}_f(p) + \lambda^{su}_f(p) = \log(d) = \lambda^s_A + \lambda^{wu}_A + \lambda^{su}_A, $ for every $p \in Per(f).$ So $\lambda^{wu}_f(p) + \lambda^{su}_f(p) = \lambda^{wu}_A + \lambda^{su}_A,$ for every $p \in Per(f).$ By Anosov Closing Lemma, $m-$almost every $x \in \mathbb{T}^3$ is such that $\lambda^{wu}_f(x) + \lambda^{su}_f(x) =  \lambda^{wu}_A + \lambda^{su}_A.$ By the Pesin formula $m$ is the measure of maximal entropy of $f,$ and then $h_{\ast}(m) = m.$

As $W^{u}_f$ is absolutely continuous, foliation the set  $Z=\{x\in{\T}^3;\, \lambda_f^u(x)=\lambda^u_f(\tilde{m})\}$ intersect $m-$a.e. unstable leaf in a positive Lebesgue measure. Moreover, since $h_{\ast}(m)=m,$ then $h$ is an absolutely continuous map.

  Applying the version of Proposition \ref{babo} for our case implies that $\lambda^{wu}_f(x) > \lambda^{wu}_A,$ for $m-$a.e $x \in \mathbb{T}^3,$ and then $W^{wu}_f$ is not absolutely continuous (see \cite{Mic22measurable}) and it implies that $h$ can not be $C^1.$

Now we show that $h$ is not Lipschitz. In fact, if $h$ was Lipschitz so \mbox{$\tilde{m} =  h_{\ast}(m)$} would be an absolutely continuous measure. Moreover, by Lemma \ref{lem3}, since $h(W^{wu}_A) = W^{wu}_f,$ $h_{\ast}$ sends conditional measures of $m$ along $wu-$leaves of $A$ to conditional measures of  $\tilde{m}$ along $wu-$leaves of $f.$ If $h$ was Lipschtzian, in particular $h_{\ast}$ should send absolutely continuous conditional measures in absolutely continuous conditional measures. Consequently $$ \displaystyle\int_{{\T}^3} \log(\jac^{wu}f(x))dm  = h_{\tilde{m}}(f, W^{wu}_f) = h_{m}(A, W^{wu}_A) = \lambda^{wu}_A.$$ Its a contradiction with $(\ref{eq01}).$ For the Pesin Formula for expanding foliations, see for instance \cite{QZ} and \cite{LQ}.

\end{proof}

\section{Proof of Theorem \ref{TeoC}}

Denote $\lambda^{wu}_i(f,x) = \lambda^{u}_i(f,x),$ $ i =1, \ldots, k $ and $\lambda^{su}_i(f,x) = \lambda^u_{k+i}(f,x),$  \mbox{$i =1, \ldots, d-1-k.$} Analogously we define $\lambda^{wu}_i(A)$ and $\lambda^{su}_i(A).$

\begin{lemma}
	Let  $\mathcal{U}$ as in Lemma \ref{lem2} and $f\in\mathcal{U}$ be a $C^2$ special and $m-$preserving Anosov endomorphism, if $W^{wu}_{1,k}$ is absolutely continuous then
	
	\begin{itemize}
		\item [a)] $\displaystyle\sum_{i = 1}^{k} \lambda^{wu}_i(f,m) =  \sum_{i = 1}^{k} \lambda^{wu}_i(A),$
		\item [b)] $\displaystyle\sum_{i = 1}^{d-1-k} \lambda^{su}_i(f, m) =  \sum_{i = 1}^{d-1-k} \lambda^{su}_i(A)$ \,\,\, and
		\item [c)] $\lambda^{s}_f( m) =  \lambda^{s}_A.$
	\end{itemize}
\end{lemma}

\begin{proof}
Consider $\tilde{f}$ and $\tilde{A}$ in the universal cover. We will start by proving items b) and a). Fist we show that $\displaystyle\sum_{i = 1}^{d-1-k} \lambda^{su}_i(f, \tilde{m}) \leq  \sum_{i = 1}^{d-1-k} \lambda^{su}_i(A)$, suppose by contradiction which there is a set $Z$ of positive volume such that for all $x\in Z$ we have $\displaystyle\sum_{i = 1}^{d-1-k} \lambda^{su}_i(f, \tilde{m}) >  \sum_{i = 1}^{d-1-k} \lambda^{su}_i(A)$. Since $f$ is $C^2$ the foliation $\widetilde{W}^u_f$ is absolutely continuous, then $Z$ intersects some leaf $\widetilde{W}^u_f(x)$ in a positive volume of the leaf. Proceeding as in the proof of Theorem \ref{TeoA}, we get a contradiction, then

\begin{equation}\label{equi07}
\displaystyle\sum_{i = 1}^{d-1-k} \lambda^{su}_i(f, \tilde{m}) \leq  \sum_{i = 1}^{d-1-k} \lambda^{su}_i(A).
\end{equation}

We are supposing that $W^{wu}_{1,k}$ is absolutely continuous. Since $W^{wu}_{1,k}$ is a uniformly expanding foliation, analogously we show that

\begin{equation}\label{equi08}
\displaystyle\sum_{i = 1}^{k} \lambda^{wu}_i(f,m) \leq  \sum_{i = 1}^{k} \lambda^{wu}_i(A).
\end{equation}

Using the Pesin entropy formula and Lemma \ref{lem1} we have,
\begin{align*}
h_m(f) &=\displaystyle\sum_{i = 1}^{k} \lambda^{wu}_i(f,m)+ \sum_{i = 1}^{d-1-k} \lambda^{su}_i(f, m)=\log(D)-\lambda_f^s(m)=\log(D)-\lambda_A^s\\
&=\displaystyle\sum_{i = 1}^{k} \lambda^{wu}_i(A)+ \sum_{i = 1}^{d-1-k} \lambda^{su}_i(A) = h_{top}(A) = h_{top}(f).
\end{align*}
where $D>1$ is the degree of $f$ and $A.$ Note that the above expression tell us that $m$ is the unique measure of maximal entropy of $f.$

Since $\displaystyle\sum_{i = 1}^{k} \lambda^{wu}_i(f,m)+ \sum_{i = 1}^{d-1-k} \lambda^{su}_i(f, m)=\displaystyle\sum_{i = 1}^{k} \lambda^{wu}_i(A)+ \sum_{i = 1}^{d-1-k} \lambda^{su}_i(A),$ by the inequalities (\ref{equi07}) and (\ref{equi08})
we obtain the equalities of the items $a)$ and $b)$ of the Lemma.

Item c) follows from Theorem 1.1 of \cite{AGGS22}.

\end{proof}	

It remains to prove the converse in Theorem \ref{TeoC}.
For this, suppose that 
\mbox{$\displaystyle\sum_{i = 1}^{k} \lambda^u_i(f, m) =  \sum_{i = 1}^{k} \lambda^u_i(A).$}

Let $\eta $ be an increasing subordinate to unstable manifold $W^u_f.$ We can consider $\eta^{wu}_f$ such that $\eta^{wu}_{f,1,k}(\tilde{x}) = \eta(\tilde{x})\cap p^{-1}(W^{wu}_{1,k}(p(\tilde{x}))).$ Since the $wu-$leaves are expanding for $f$ and submanifolds of $W^u_{f},$ then the partition $\eta^{wu}_{f,1,k}$ is an increasing subordinate to $W^{wu}_{1,k}$ partition, the scenery is done. Now the theory of SRB of \cite{QZ} can be applied to $W^{wu}_{1,k}.$ Denote by $\eta^{wu}_{A,1,k} = \tilde{h}^{-1}(\eta^{wu}_{f,1,k}),$ where $\tilde{h}$ is the induced by $h$ in the level of limit inverse spaces. By the increasing property of subordinate partitions, the values of conditional entropies
$H_{\hat{m}}(\eta^{wu}_{A,1,k} | \hat{A}(\eta^{wu}_A) ), H_{\hat{m}}( \eta^{wu}_{f,1,k} |\hat{f}(\eta^{wu}_{f,1,k})) $ are independent of the chosen increasing subordinated to $wu-$foliation partition, see for instance Lemma 5.3, Chapter VI of \cite{LQ}. These numbers we call respectively by $ h_{m}(A,W^{wu}_{1,k}(A))$ and $h_{m}(f,W^{wu}_{1,k}).$

Again, we know that $\tilde{m} = h_{\ast}(m)$ is the unique measure of maximal entropy for $f,$ so $m = \tilde{m}.$ Since $m = h_{\ast}(m),$ by Lemmas \ref{lem2} and \ref{lem3}
$$ h_{m} (f, W^{wu}_{1,k}) = h_{m} (A, W^{wu}_{1,k}(A)) = \sum_{i = 1}^{k} \lambda^u_i(A) =  \displaystyle\sum_{i = 1}^{k} \lambda^u_i(f, m), $$
the above expression says us that the unstable entropy of $f$ along the expanding foliation $W^{wu}_{1,k}$ satisfies the Pesin formula, so $m$ disintegrates as Lebesgue along the leaves of that foliation, see \cite{QZ} and \cite{LQ}, for Pesin Formula for expanding foliations and consequently the foliation $W^{wu}_{1,k}$ is absolutely continuous.


%

\begin{thebibliography}{RRRRRR}

\bibitem{AGGS22}
An, J., Gan, S., Gu, R., Shi, Y.
  \newblock  Rigidity of stable Lyapunov exponents and integrability for Anosov maps.
  \newblock {\em arXiv preprint} arXiv:2205.13144, 2022.


\bibitem{AC22}
Álvarez, C. F.,  Cantarino, M.
\newblock Existence and uniqueness of measures of maximal entropy for partially hyperbolic endomorphisms.
\newblock {\em  arXiv preprint} arXiv:2207.05823, 2022.








\bibitem{BB}
A. Baraviera and C. Bonatti,
\newblock Removing zero Lyapunov Exponents.
\newblock {\em Ergodic Theory of Dynamical Systems}, 23 :1655--1670, 2003.


\bibitem{B}
 M. Brin,
\newblock On dynamical coherence.
\newblock {\em Ergodic theory and dynamical systems,} vol. 23, No. 2, 395--401, Cambridge Univ Press, 2003.
	



\bibitem{CM19}
J. S. Costa, F. Micena,
\newblock Pathological center foliation with dimension greater than one.
\newblock {\em Discrete and Continuous Dynamical Systems}  v. 39, n. 2, (2019), 1049-1070.


\bibitem{CM22}
J. S. Costa, F. Micena,
\newblock Some generic properties of partially hyperbolic endomorphisms.
\newblock {\em Nonlinearity}  v. 35,  (2022), 5297–5310.


\bibitem{CT23} 
J. S. Costa, A. Tahzibi,
\newblock Semi rigidity of Lyapunov exponents in higher dimensions (in preparation).




\bibitem{F70}
J. Franks,
\newblock Anosov diffeomorphisms, Global Analysis
\newblock {\em Proc. Sympos. Pure Math.}  14 (1970), 61-93.






\bibitem{H}
A. Hammerlindl,
\newblock Leaf Conjugacies on the Torus.
\newblock {\em Ergodic Theory and Dynamical Systems,} 33 (2013), no. 3, 896–-933.


\bibitem{LQ}
P-D. Liu, M. Qian
\newblock Smooth Ergodic Theory of Random Dynamical Systems.
\newblock {\em Lecture Notes in Mathematics 1606,} Springer, 1995.















\bibitem{MT13}
F. Micena, A. Tahzibi,
\newblock Regularity of foliations and Lyapunov exponents for partially hyperbolic Dynamics.
\newblock {\em Nonlinearity,}  (2013), no. 33, 1071--1082.










































































\bibitem{MP75}
R. Ma\~{n}\'{e}, C. Pugh,
\newblock Stability of endomorphisms.
\newblock {\em Warwick Dynamical Systems
	1974.} Lecture Notes in Math., 468, Springer, 1975,  175--184.


\bibitem{M74}
Manning, A.
\newblock There are no new Anosov diffeomorphisms on tori.
\newblock {\em American Journal of Mathematics}, v.96 n.3 (1974): 422-429.









\bibitem{Mic22measurable}
F. Micena.
\newblock On Measurable Properties of Anosov Endomorphisms of Torus.
\newblock	arXiv preprint. arXiv: 2207.13986,
2022


\bibitem{MT16}
F. Micena, A. Tahzibi,
\newblock On the Unstable Directions and Lyapunov Exponents of Anosov Endomorphisms.
\newblock {\em Fund. Math.,}  (2016), no. 235, 37--48.

\bibitem{MT19}
Moosavi, S. M.,  Tajbakhsh, K.
\newblock Classification of special Anosov endomorphisms of nil-manifolds.
\newblock {\em  Acta Mathematica Sinica, English Series}, v. 35, n. 12, p. 1871-1890, 2019.










\bibitem{PDL1}
P-D. Liu,
\newblock Pesin's Entropy formula for endomorphisms.
\newblock {\em Nagoya Math. J.}, 150: 197--209, 1998.

\bibitem{pesin2004lectures}
Y. Pesin,
\newblock Lectures on Partial Hyperbolicity and Stable Ergodicity.
\newblock {\em European Mathematical Society,} 2004.

\bibitem{PRZ}
F. Przytycki,
\newblock Anosov endomorphisms.
\newblock {\em Studia Math.,} 58 (1976) :249--285.

\bibitem{QXZ}
M. Qian, J-S. Xie; S. Zhu,
\newblock Smooth ergodic theory for endomorphisms.
\newblock{Lecture notes in mathematics,} Vol. 1978. Springer-Verlag, Berlin Heidelberg, 2009.

\bibitem{QZ}
M. Qian, S. Zhu,
\newblock SRB measures and Pesin's entropy formula for endomorphisms.
\newblock {\em Trans. Am.
Math. Soc.,} 354(4) (2002) :1453--1471.






\bibitem{S87}
K. Sakai,
\newblock Anosov maps on closed topological manifolds.
\newblock {\em J. Math. Soc. Japan,} 39 (1987), 505--519.



\bibitem{sumi}
N. Sumi,
\newblock Topological Anosov maps of infra-nil-manifolds.
\newblock {\em J. Math. Soc. Japan,}
Vol. 48, No. 4, 607--648, 1996.









\end{thebibliography}


\end{document}